\documentclass[12pt,reqno]{amsart}

\usepackage{graphicx}
\usepackage[top=20mm, bottom=20mm, left=30mm, right=30mm]{geometry}
\usepackage[colorlinks=true,citecolor=black,linkcolor=black,urlcolor=magenta]{hyperref}

\usepackage{fancyvrb}
\usepackage{ulem} 
\usepackage[T1]{fontenc} 
\usepackage{pbsi} 
\usepackage{url} 

\usepackage{local-math-commands}

\usepackage{diagbox}
\newcommand{\trianglenk}[2]{$\diagbox{#1}{#2}$}

\title[Generating Function Transformations]{Zeta Series 
       Generating Function Transformations Related to 
       Polylogarithm Functions and the $k$--Order Harmonic Numbers}  

\author{Maxie D. Schmidt \\ 
        \href{mailto:maxieds@gmail.com}{maxieds@gmail.com}}         

\allowdisplaybreaks

\begin{document}

\maketitle

\begin{abstract} 
We define a new class of generating function transformations related to 
polylogarithm functions, Dirichlet series, and Euler sums. 
These transformations are given by an infinite sum over the $j^{th}$ 
derivatives of a sequence generating function and sets of 
generalized coefficients satisfying a non--triangular recurrence relation 
in two variables. The generalized transformation coefficients share a 
number of analogous properties with the Stirling numbers of the second kind 
and the known harmonic number expansions of the 
unsigned Stirling numbers of the first kind. 

We prove a number of properties of the generalized coefficients which 
lead to new recurrence relations and summation identities for the 
$k$--order harmonic number sequences. 
Other applications of the generating function transformations we 
define in the article include new series expansions for the 
polylogarithm function, the alternating zeta function, and the 
Fourier series for the periodic Bernoulli polynomials. 
We conclude the article with a discussion of several specific new 
``almost'' linear recurrence relations between the integer--order harmonic 
numbers and the generalized transformation coefficients, which 
provide new applications to studying the limiting behavior of the 
zeta function constants, $\zeta(k)$, at integers $k \geq 2$. 
\end{abstract}

\section{Introduction}

The \textit{Stirling numbers of the second kind}, 
$\gkpSII{n}{k}$, are defined for $n,k \geq 0$ by the 
triangular recurrence relation  \citep[\Section 6.1]{GKP}\footnote{ 
     The notation for \textit{Iverson's convention}, 
     $\Iverson{n = k} = \delta_{n,k}$, is consistent with its usage in 
     \citep{GKP}. 
}
\begin{equation} 
\label{eqn_S2_rdef} 
\gkpSII{n}{k} = k \gkpSII{n-1}{k} + \gkpSII{n-1}{k-1} + \Iverson{n = k = 0}.
\end{equation} 
It is also known, or at least straightforward to prove by induction, that 
for any sequence, $\langle g_n \rangle$, whose formal ordinary power series 
(OGF) is denoted by $G(z)$, and natural numbers $m \geq 1$, 
we have a generating function transformation of the form 
\citep[\cf \S 7.4]{GKP}\footnote{ 
     Variants of \eqref{eqn_S2NPowCoeffGFTransform} can be found in 
     \citep[\S 26.8(v)]{NISTHB}.
     A special case of the identity for 
     $f_n \equiv 1$ appears in \citep[eq. (7.46); \S 7.4]{GKP}. 
     Other related expansions for converting between powers of the 
     differential operator $D$ and the operator $\vartheta := zD$ are 
     known as sums involving the Stirling numbers of the first and 
     second kinds \citep[Ex. 6.13; \cf \S 6.5]{GKP}. The particular identity 
     that $[z^n]\left((zD)^k F(z)\right) = n^k f_n$ is stated 
     in \citep[\S 2.2]{GFOLOGY}. 
} 
\begin{equation}
\label{eqn_S2NPowCoeffGFTransform}
\sum_{n \geq 0} n^m g_n z^n = \sum_{j=0}^{m} \gkpSII{m}{j} z^j G^{(j)}(z).
\end{equation}
We seek to study the properties of a related set of coefficients that 
provide the corresponding semi ''inverse'' generating function transformations 
of the form 
\begin{equation} 
\label{eqn_S2StarCoeffGFTransform}
\sum_{n \geq 1} \frac{g_n}{n^k} z^n = 
     \sum_{j \geq 1} \SIIStarCf{k+2}{j} z^j G^{(j)}(z), 
\end{equation} 
for integers $k > 0$. 
We readily see that the generalized coefficients, $\SIIStarCf{k}{j}$, 
are defined by a two--index, non--triangular recurrence relation of the 
form
\begin{align}
\label{eqn_S2StarCfRecRelation}
\SIIStarCf{k}{j} & = -\frac{1}{j} \SIIStarCf{k}{j-1} + 
     \frac{1}{j} \SIIStarCf{k-1}{j} + \Iverson{k=j=1} \\ 
\notag 
     & = 
     \sum_{1 \leq m \leq j} \binom{j}{m} \frac{(-1)^{j-m}}{j! m^{k-2}}, 
\end{align}
which provides a number of new properties, identities, and sequence 
applications involving these numbers. 

We likewise obtain a number of new, interesting relations between the 
\textit{$r$--order harmonic numbers}, $H_n^{(r)} = \sum_{k=1}^{n} k^{-r}$ and 
$H_n^{(r)}(t) = \sum_{k=1}^{n} t^k / k^r$, by 
their corresponding ordinary generating functions, 
$\Li_{r}(z) / (1-z)$ and $\Li_{r}(tz) / (1-z)$, 
through our study of the generalized transformation coefficients, 
$\SIIStarCf{k}{j}$, in \eqref{eqn_S2StarCfRecRelation}. 
Most of the series expansions for special functions we define through 
\eqref{eqn_S2StarCoeffGFTransform} are new, and moreover, 
provide rational partial series approximations to the infinite series in $z$. 
Section \ref{subSection_InitProps_OGFs_S1Relations_of_GenCoeffs-proof} 
provides the details to a combinatorial proof of the 
zeta series transformations defined by 
\eqref{eqn_S2StarCoeffGFTransform} and \eqref{eqn_S2StarCfRecRelation}. 

\subsection*{Examples} 

The Dirichlet--generating--function--like series defined formally by 
\eqref{eqn_S2StarCoeffGFTransform} can be approximated up to any finite 
order $u \geq 1$ by the terms of typically rational truncated Taylor series. 
We cite a few notable examples of these truncated ordinary generating 
functions in the following equations where $k \in \mathbb{N}$, 
$a, b, r, t \in \mathbb{R}$, and $\omega_a = \exp(2\pi\imath / a)$ denotes the 
\textit{primitive $a^{th}$ root of unity}:
\begin{subequations}
\begin{align}
\sum_{1 \leq n \leq u} \frac{z^n}{n^k} & = 
     [w^u]\left(\sum_{j=1}^u \SIIStarCf{k+2}{j} 
     \frac{(wz)^j j!}{(1-wz)^{j+1} (1-w)}\right) \\ 
\sum_{1 \leq n \leq u} \frac{z^n}{n^k n!} & =  
     [w^u]\left(\sum_{j=1}^u \SIIStarCf{k+2}{j} 
     \frac{(wz)^j e^{wz}}{(1-w)}\right)\\ 
\sum_{1 \leq n \leq u} H_n^{(k)} z^n & =  
     [w^u]\left(\sum_{j=1}^u \SIIStarCf{k+2}{j} 
     \frac{(wz)^j j!}{(1-wz)^{j+2} (1-w)}\right) \\ 
\sum_{1 \leq n \leq u} \left(\sum_{m=1}^{n} \frac{t^m}{m^k}\right) z^n & =  
     [w^u]\left(\sum_{j=1}^u \SIIStarCf{k+2}{j} 
     \frac{(wtz)^j j!}{(1-wtz)^{j+1} (1-wz) (1-w)}\right) \\ 
\sum_{1 \leq n \leq u} \left(\sum_{m=1}^{n} \frac{r^m}{m^k m!}\right) z^n & = 
     [w^u]\left(\sum_{j=1}^u \SIIStarCf{k+2}{j} 
     \frac{(wrz)^j e^{wrz}}{(1-wz) (1-w)}\right) \\ 
\label{eqn_HNumEGFFromOGFIntTransform}
\sum_{1 \leq n \leq u} \frac{H_n^{(k)}}{n!} z^n & = 
     [w^u]\left(\sum_{j=1}^u \SIIStarCf{k+2}{j} 
     \frac{(wz)^j}{(j+1)}\frac{e^{wz}(j+1+wz)}{(1-w)}\right) \\ 
\sum_{1 \leq n \leq u} \frac{z^n}{(an+b)^{s}} + 
     \frac{\Iverson{b > 0}}{b^s} & = 
     [t^{au+b}]\left(\sum_{m=0}^{a-1} \sum_{j=1}^{au+b} \SIIStarCf{s+2}{j} 
     \frac{\omega_a^{-mb} z^{b/a} (\left(t z^{1/a}\right)^j j!}{a 
     \left((1-\omega_a^m t z^{1/a}\right)^{j+1} (1-t)}\right).
\end{align} 
\end{subequations}
     These expansions follow easily as consequences of a few generating 
     function operations and transformation results. First, 
     for any fixed scalar $t$, 
     the $j^{th}$ derivative of the geometric series satisfies 
     \begin{equation*} 
         \frac{d^{(j)}}{{dz}^{(j)}}\Biggl[\frac{1}{(1-tz)}\Biggr] = 
         \frac{t^j j!}{(1-tz)^{j+1}}. 
     \end{equation*} 
     We also have an integral transform that converts the ordinary 
     generating function, $F(z)$, of any sequence into its corresponding 
     exponential generating function, $\widehat{F}(z)$, according to 
     \citep[p.\ 566]{GKP} 
     \begin{equation*} 
     \widehat{F}(z) = \frac{1}{2\pi} \int_{-\pi}^{+\pi} F\left( 
          z e^{-\imath\vartheta}\right) e^{e^{\imath\vartheta}} d\vartheta. 
     \end{equation*} 
     This integral transformation together with induction and 
     integration by parts shows that 
     \begin{equation*} 
     \frac{1}{2\pi} \int_{-\pi}^{+\pi} \frac{\left(wz 
          e^{-\imath\vartheta}\right)^{j}}{\left(1-wz e^{-\imath\vartheta} 
          \right)^{j+2}} e^{e^{\imath\vartheta}} d\vartheta = 
          \frac{(wz)^j e^{wz}}{(j+1)!}\left(j+1+wz\right), 
     \end{equation*} 
     which implies the second to last expansion in 
     \eqref{eqn_HNumEGFFromOGFIntTransform}. 
     Lastly, there is a known ``\textit{series multisection}'' 
     generating function transformation over 
     arithmetic progressions of a sequence for integers 
     $a > 1, b \geq 0$ of the form \citep[\S 1.2.9]{TAOCPV1}
     \begin{equation*}
     \sum_{n \geq 0} f_{an+b} z^{an+b} = 
          \sum_{0 \leq m < a} 
          \frac{\omega_a^{-mr}}{a} F\left(\omega_a^m z\right)
     \end{equation*}
Since the geometric series ordinary generating function, and its $j^{th}$ 
derivatives, are always rational, we may also give similar statements 
about the partial sums of the \textit{Euler sum} generating functions 
of the forms studied in 
\citep{FLAJOLETESUMS,EXPLICIT-EVAL-ESUMS,ECV2}. 

\subsection*{Comparison to Known Series} 

For comparison, we summarize a pair of known series identities for the 
\textit{polylogarithm function}, $\Li_s(z)$, and the 
\textit{modified Hurwitz zeta function,} 
$\Phi(z, s, \alpha, \beta) = \sum_{n \geq 1} z^n / (\alpha n + \beta)^s = 
 \alpha^{-s} z \Phi(z, s, \beta / \alpha + 1)$, as follows 
\citep[\S 25.12(ii), 25.14]{NISTHB} 
\citep[eq. (6); Thm. 2.1; \S 2]{DOUBLEINTS-LTRANS}: 
\begin{align} 
\label{eqn_Lisz_ModHZetaFn_KnownSeriesExps}
\Li_s(z) & = \sum_{k \geq 0} \left(-\frac{z}{1-z}\right)^{k+1} 
     \sum_{0 \leq m \leq k} \binom{k}{m} \frac{(-1)^{m+1}}{(m+1)^{s}} \\ 
\notag
\Phi(z, s, \alpha, \beta) & =\sum_{k \geq 0} \left(-\frac{z}{1-z}\right)^{k+1} 
     \sum_{0 \leq m \leq k} \binom{k}{m} 
     \frac{(-1)^{m+1}}{(\alpha m+\alpha+\beta)^{s}}. 
\end{align} 
The new generalized coefficients, $\SIIStarCf{k}{j}$, 
defining the transformations in \eqref{eqn_S2StarCoeffGFTransform} 
satisfy several key properties and generating functions analogous to 
those of the 
Stirling numbers of the second kind, and are closely--related to the 
known harmonic number expansions of the unsigned triangle of the 
\textit{Stirling numbers of the first kind} 
\citep{STIRESUMS,MULTIFACTJIS}. 
We explore several initial properties and relations of these 
coefficients in the next section. 

\section{Initial Properties, Ordinary Generating Functions, and 
         Relations to the Stirling Numbers} 
\label{Section_InitProps_OGFs_S1Relations_of_GenCoeffs} 

\subsection{Proof of the Transformation Identity in 
            \textbf{\eqref{eqn_S2StarCoeffGFTransform}}} 
\label{subSection_InitProps_OGFs_S1Relations_of_GenCoeffs-proof} 

We first prove that the recurrence relation in \eqref{eqn_S2StarCfRecRelation} 
holds for the generalized transformation coefficients in 
\eqref{eqn_S2StarCoeffGFTransform}, which is then used to extrapolate new 
results providing summation and harmonic number identities for these 
sequences. 
\begin{proof}[Proof of \eqref{eqn_S2StarCfRecRelation}]
The proof proceeds by an inductive argument similar to the proof that can 
be given from \eqref{eqn_S2_rdef} for the 
generating function transformations involving the Stirling numbers of the 
second kind cited in the introduction. 
We first observe that 
\begin{equation*}
\sum_{n \geq 1} \frac{g_n}{n^k} z^n = 
     \sum_{n \geq 1} \frac{(n \cdot g_n)}{n^{k+1}} z^n 
\end{equation*} 
for all $k \in \mathbb{N}$. 
Since the ordinary generating function for the sequence, 
$\langle n g_n \rangle$, is given by $z G^{\prime}(z)$, and the $j^{th}$ 
derivative of $z G^{\prime}(z)$ is $j G^{(j)}(z) + z G^{(j+1)}(z)$, 
we may write that 
\begin{align*} 
\sum_{j \geq 1} \SIIStarCf{k+2}{j} z^j G^{(j)}(z) & = 
     \sum_{j \geq 1} \SIIStarCf{k+3}{j} z^j \left( 
     j G^{(j)}(z) + z G^{(j+1)}(z) \right) \\ 
     & = 
     \sum_{j \geq 2} z^j \left(j \SIIStarCf{k+3}{j} + \SIIStarCf{k+3}{j-1} 
     \right) + \Iverson{k = j = 1}. 
\end{align*} 
We then conclude that the non--triangular recurrence relation in 
\eqref{eqn_S2StarCfRecRelation} defines the series transformation 
coefficients in \eqref{eqn_S2StarCoeffGFTransform}. 
\end{proof} 

\subsection{Exact Expansions of the 
            Transformation Coefficients} 

\begin{table}[t] 
\renewcommand{\arraystretch}{1.25} 
\setlength{\arraycolsep}{3pt} 
\begin{equation*} 
\begin{array}{|c|lllllllll|} \hline 
\trianglenk{k}{j} & 0 & 1 & 2 & 3 & 4 & 5 & 6 & 7 & 8 \\ \hline 
0 & 1 & 0 & 0 & 0 & 0 & 0 & 0 & 0 & 0 \\ 
1 & 0 & 1 & 0 & 0 & 0 & 0 & 0 & 0 & 0 \\ 
2 & 0 & 1 & -\frac{1}{2} & \frac{1}{6} & -\frac{1}{24} & \frac{1}{120} & 
    -\frac{1}{720} & \frac{1}{5040} & -\frac{1}{40320} \\ 
3 & 0 & 1 & -\frac{3}{4} & \frac{11}{36} & -\frac{25}{288} & \frac{137}{7200} & 
        -\frac{49}{14400} & \frac{121}{235200} & -\frac{761}{11289600} \\ 
4 & 0 & 1 & -\frac{7}{8} & \frac{85}{216} & -\frac{415}{3456} & 
         \frac{12019}{432000} & -\frac{13489}{2592000} & 
         \frac{726301}{889056000} & -\frac{3144919}{28449792000} \\ 
5 & 0 & 1 & -\frac{15}{16} & \frac{575}{1296} & -\frac{5845}{41472} & 
        \frac{874853}{25920000} & -\frac{336581}{51840000} & 
        \frac{129973303}{124467840000} & -\frac{1149858589}{7965941760000} \\ 
6 & 0 & 1 & -\frac{31}{32} & \frac{3661}{7776} & -\frac{76111}{497664} & 
        \frac{58067611}{1555200000} & -\frac{68165041}{9331200000} & 
        \frac{187059457981}{156829478400000} &
        -\frac{3355156783231}{20074173235200000} \\ \hline 
\end{array} 
\end{equation*} 
\caption{A Table of the Generalized Coefficients $\SIIStarCf{k}{j}$} 
\label{table_TableOfGenCoeffs}
\end{table} 

\begin{table}[ht] 
\renewcommand{\arraystretch}{1.25} 
\setlength{\arraycolsep}{4pt} 
\begin{equation*} 
\begin{array}{|c|lllllllll|} \hline 
\trianglenk{k}{j} & 0 & 1 & 2 & 3 & 4 & 5 & 6 & 7 & 8 \\ \hline 
0 & 1 & 0 & 0 & 0 & 0 & 0 & 0 & 0 & 0 \\ 
1 & 0 & 1 & 0 & 0 & 0 & 0 & 0 & 0 & 0 \\
2 & 0 & 1 & 1 & 1 & 1 & 1 & 1 & 1 & 1 \\
3 & 0 & 1 & \frac{3}{2} & \frac{11}{6} & \frac{25}{12} & \frac{137}{60} & 
     \frac{49}{20} & \frac{363}{140} & \frac{761}{280} \\
4 & 0 & 1 & \frac{7}{4} & \frac{85}{36} & \frac{415}{144} & \frac{12019}{3600} & 
     \frac{13489}{3600} & \frac{726301}{176400} & \frac{3144919}{705600} \\
5 & 0 & 1 & \frac{15}{8} & \frac{575}{216} & \frac{5845}{1728} & 
     \frac{874853}{216000} & \frac{336581}{72000} & 
     \frac{129973303}{24696000} & \frac{1149858589}{197568000} \\
6 & 0 & 1 & \frac{31}{16} & \frac{3661}{1296} & \frac{76111}{20736} & 
     \frac{58067611}{12960000} & \frac{68165041}{12960000} & 
     \frac{187059457981}{31116960000} &
     \frac{3355156783231}{497871360000} \\ \hline 
\end{array}
\end{equation*} 
\caption{A Table of the Scaled Coefficients 
         $\SIIStarCf{k}{j} \times (-1)^{j-1} \cdot j!$} 
\label{table_TableOfGenCoeffs_scaled_v2}
\end{table} 

The recurrence relation in \eqref{eqn_S2StarCfRecRelation} 
leads us to compute the first few terms of these sequences given in 
Table \ref{table_TableOfGenCoeffs} and 
Table \ref{table_TableOfGenCoeffs_scaled_v2}. 
We are also able to compute the next explicit formulas for 
variable $k$ and fixed small special cases of $j \geq 1$ as follows: 
\begin{align} 
\label{prop_Several_SpCase_Formulas_inj_GenCoeffs} 
\SIIStarCf{k}{1} & = \phantom{-} \Iverson{k \geq 1} \\ 
\notag 
\SIIStarCf{k}{2} & = -\left(1-2^{1-k}\right) \Iverson{k \geq 2} \\ 
\notag
\SIIStarCf{k}{3} & = \phantom{-} \frac{1}{2}\left(1-2 \cdot 2^{1-k} + 
     3^{1-k}\right) \Iverson{k \geq 2} \\ 
\notag
\SIIStarCf{k}{4} & = -\frac{1}{6}\left(1-3 \cdot 2^{1-k} + 
     3 \cdot 3^{1-k}-4^{1-k}\right) \Iverson{k \geq 2} \\ 
\notag
\SIIStarCf{k}{5} & = \phantom{-} \frac{1}{24}\left(1-4 \cdot 2^{1-k} + 
     6 \cdot 3^{1-k}-4\cdot 4^{1-k}+5^{1-k}\right) \Iverson{k \geq 2}. 
\end{align} 
The inductive proof of the full explicit summation formula expanded in 
\eqref{eqn_S2Star_coeff_closed-form_sum_ident} 
we obtain from the special cases above 
is left as an exercise to the reader. 
We compare this formula to the analogous identity for the 
Stirling numbers of the second kind as follows 
\citep[\S 26.8]{NISTHB}: 
\begin{align} 
\label{eqn_S2_coeff_closed-form_sum_ident} 
\gkpSII{k}{j} & = \sum_{m=1}^{j} \binom{j}{m} 
     \frac{(-1)^{j-m} m^{k}}{j!} \\ 
\label{eqn_S2Star_coeff_closed-form_sum_ident} 
\SIIStarCf{k+2}{j} & = \sum_{m=1}^{j} \binom{j}{m} 
     \frac{(-1)^{j-m}}{j!\ m^{k}} \phantom{m^{k}}. 
\end{align} 
For further comparison, observe that the forms of both 
\eqref{eqn_S2_coeff_closed-form_sum_ident} and 
\eqref{eqn_S2Star_coeff_closed-form_sum_ident} lead to the following 
similar pair of ordinary generating functions in $z$ with respect to the 
upper index $k > 0$ and fixed $j \in \mathbb{Z}^{+}$ 
\citep[\Section 26.8(ii)]{NISTHB}: 
\begin{align} 
\label{eqn_S2_OGFjz_GF_ident} 
\sum_{k=0}^{\infty} \gkpSII{k}{j} z^k & = 
     \frac{z^j}{(1-z)(1-2z) \cdots (1-jz)} \\ 
\notag 
\sum_{k=0}^{\infty} \SIIStarCf{k}{j} z^k & = 
     \left(\frac{(-1)^{j+1} z^2}{(1-z)(2-z) \cdots (j-z)}\right) 
     \Iverson{j \geq 2} + 
     \left(\frac{z}{(1-z)}\right) \Iverson{j = 1}. 
\end{align} 
We also compare the generalized coefficient formula in 
\eqref{eqn_S2Star_coeff_closed-form_sum_ident} and 
its generating function representation in 
\eqref{eqn_S2_OGFjz_GF_ident} 
to the \textit{N\"{o}rlund--Rice integral} of a meromorphic function $f$ 
over a suitable contour given by \citep{FS-NORLUND-RICE-INTS}
\begin{equation*}
\sum_{1 \leq m \leq j} \binom{j}{m} (-1)^{j-m} f(m) = 
     \frac{j!}{2\pi\imath} \oint 
     \frac{f(z)}{z(z-1)(z-2) \cdots (z-j)} dz. 
\end{equation*}

\begin{cor}[Harmonic Number Formulas] 
\label{cor_Summary_S2Star_Hnr_HarmonicNumber_exps_k_in_2to6} 
For $j \in \mathbb{N}$, the following formulas provide 
expansions of the coefficients from 
\eqref{eqn_S2StarCfRecRelation} and 
\eqref{eqn_S2Star_coeff_closed-form_sum_ident} at the 
fixed cases of $k \in [2, 6] \subseteq \mathbb{N}$: 
\begin{align} 
\notag 
\SIIStarCf{2}{j} & = \frac{(-1)^{j-1}}{j!} \\ 
\notag 
\SIIStarCf{3}{j} & = \frac{(-1)^{j-1}}{j!}\ H_j \\ 
\notag 
\SIIStarCf{4}{j} & = \frac{(-1)^{j-1}}{2 j!}\left(H_j^2 + H_j^{(2)}\right) \\ 
\notag 
\SIIStarCf{5}{j} & = \frac{(-1)^{j-1}}{6 j!}\left(H_j^3 + 3 H_j H_j^{(2)} + 
     2 H_j^{(3)}\right) \\ 
\label{eqn_S2Star_Hnr_HarmonicNumber_exps_k_in_2to6} 
\SIIStarCf{6}{j} & = \frac{(-1)^{j-1}}{24 j!}\left( 
     H_j^4 + 6 H_j^2 H_j^{(2)} + 3 \left(H_j^{(2)}\right)^2 + 
     8 H_j H_j^{(3)} + 6 H_j^{(4)}\right).
\end{align} 
\end{cor} 
\begin{proof} 
Both the \textit{Mathematica} software suite and the package 
\texttt{Sigma}\footnote{ 
     \url{https://www.risc.jku.at/research/combinat/software/Sigma/}
} 
are able to obtain these formulas for small special cases 
\citep{SYMB-SUM-COMB-SIGMAPKGDOCS}. 
Larger special cases of $k \geq 7$ are easiest to compute by first 
generating a recurrence corresponding to the sum in 
\eqref{eqn_S2Star_coeff_closed-form_sum_ident}, 
and then solving the resulting non--linear recurrence relation 
with the \texttt{Sigma} package routines. 

A more general heuristic harmonic--number--based 
recurrence formula that generates 
these expansions for all $k \geq 2$ is suggested along the 
lines of the analogous formulas for the Stirling numbers of the first kind 
in the references as \citep[\S 2]{STIRESUMS} 
\begin{equation} 
\label{eqn_S2StarCf_HNum_Heuristic_formula_stmt} 
\SIIStarCf{k+2}{j} = \sum_{0 \leq m < k} \frac{H_j^{(m+1)}}{k} 
     \SIIStarCf{k+1-m}{j} + \left(\frac{(-1)^{j-1}}{j!}\right) 
     \Iverson{k = 0}. 
\end{equation}
A short proof of the identity in 
\eqref{eqn_S2StarCf_HNum_Heuristic_formula_stmt} for all $k \geq 1$ 
is given through the exponential of a generating function for the 
$r$--order harmonic numbers and properties of the 
\emph{Bell polynomials}, $Y_n(x_1, x_2, \ldots, x_n)$, in 
\citep{CONNON-BELL-POLY-APPS}. 
\end{proof} 

\begin{table}[h] 
\begin{tabular}{|r|l|l|} \hline 
$k$ & $t_0^{(k)}(j)$ & $t_1^{(k)}(j)$ \\ \hline 
$$2 & $0$ & $2$ \\ 
$3$ & $0$ & $2 H_j$ \\ 
$4$ & $H_j^{(2)}$ & $H_j^2$ \\ 
$5$ & $H_j H_j^{(2)}$ & $\frac{1}{3}\left(H_j^3+2 H_j^{(3)}\right)$ \\ 
$6$ & $\frac{1}{2}\left(H_j^2 H_j^{(2)}+H_j^{(4)}\right)$ & 
     $\frac{1}{12}\left(H_j^4+3 \left(H_j^{(2)}\right)^2 + 
      8 H_j H_j^{(3)}\right)$ \\ 
$7$ & $\frac{1}{6}\left(H_j^3 H_j^{(2)}+2 H_j^{(2)} H_j^{(3)} + 
       3 H_j H_j^{(4)}\right)$ & 
     $\frac{1}{60}\left(H_j^5+15 H_j \left(H_j^{(2)}\right)^2+20 H_j^2 
      H_j^{(3)}+24 H_j^{(5)}\right)$ \\ \hline 
\end{tabular} 
\vskip 0.1in 
\caption{The Harmonic Number Remainder Terms in 
         \eqref{eqn_t0kj_t1kj_fn_S2S-S1-HNum_RemainderTerms_defs_v1} and 
         \eqref{eqn_t0kj_t1kj_fn_S2S-S1-HNum_RemainderTerms_defs_v2}} 
\label{table_t01kj_HNum_Remainder_Terms_formulas} 
\end{table} 

\begin{example}[Comparison of the Formulas] 
The similarities between the harmonic number expansions of the 
Stirling numbers of the first kind and the related expansions of
\eqref{eqn_S2Star_coeff_closed-form_sum_ident} given in 
\eqref{eqn_S2Star_Hnr_HarmonicNumber_exps_k_in_2to6} 
suggest another interpretation for the generalized 
coefficient representations. 
More precisely, for $k \in \mathbb{N}$, let the respective functions 
$t_0^{(k)}(j)$ and $t_1^{(k)}(j)$ denote the 
remainder terms in the forms of the coefficients in 
\eqref{eqn_S2Star_coeff_closed-form_sum_ident} 
defined by the following pair of equations:
\begin{align} 
\label{eqn_t0kj_t1kj_fn_S2S-S1-HNum_RemainderTerms_defs_v1} 
t_0^{(k+2)}(j) & = \SIIStarCf{k+2}{j} \cdot (-1)^{j-1} \cdot j! - 
     \gkpSI{j+1}{k+1} \frac{1}{j!} \\ 
\notag 
t_1^{(k+2)}(j) & = \SIIStarCf{k+2}{j} \cdot (-1)^{j-1} \cdot j! + 
     \gkpSI{j+1}{k+1} \frac{1}{j!}. 
\end{align} 
The harmonic number formulas for these generalized coefficient sums 
are recovered from the remainder terms in 
\eqref{eqn_t0kj_t1kj_fn_S2S-S1-HNum_RemainderTerms_defs_v1}
and the Stirling number expansions as 
\begin{align} 
\label{eqn_t0kj_t1kj_fn_S2S-S1-HNum_RemainderTerms_defs_v2} 
\SIIStarCf{k+2}{j} & = \frac{(-1)^{j-1}}{j!} \left( 
     \gkpSI{j+1}{k+1} \frac{1}{j!} + t_0^{(k+2)}(j) 
     \right) \\ 
\notag 
\SIIStarCf{k+2}{j} & = \frac{(-1)^{j\phantom{-1}}}{j!} \left( 
     \gkpSI{j+1}{k+1} \frac{1}{j!} - t_1^{(k+2)}(j) 
     \right). 
\end{align} 
The heuristic method identified in the proof of 
Corollary \ref{cor_Summary_S2Star_Hnr_HarmonicNumber_exps_k_in_2to6}
allows for the form of both functions to be computed for the 
next several special case formulas extending the expansions 
cited in the corollary. 
For comparison, a table of the first several of these 
remainder functions is provided in 
Table \ref{table_t01kj_HNum_Remainder_Terms_formulas} for 
$2 \leq k \leq 7$. 
\end{example} 

\section{Recurrence Relations and Other Identities for 
         Harmonic Number Sequences} 
\label{subSection_Apps_RecRels_Idents_HNumSeqs} 
\label{subSection_Additional_Props_Exps_of_the_GenCoeffs} 

\subsection{Finite Sum Expansions of the Transformation Coefficients} 

\begin{prop}[Integer--Order Harmonic Number Identities] 
\label{prop_Integer-Order_HNumber_idents_v1} 
For natural numbers $k \geq 1$, the generalized coefficients in 
\eqref{eqn_S2StarCfRecRelation} satisfy the following identities: 
\begin{align} 
\label{eqn_Integer-Order_HNumber_idents_v1-first_sum_v1} 
\SIIStarCf{k+2}{j} & = (j+1) \sum_{i=0}^{j-1} \frac{(-1)^{j-1-i} 
     H_{i+1}^{(k)}}{(j-1-i)! (i+2)!} \\ 
\notag 
\SIIStarCf{k+2}{j} & = (j+1) \sum_{i=0}^{j-1} 
     \frac{(-1)^{j-1-i}}{(j-1-i)!}\left( 
     \frac{H_{i+2}^{(k)}}{(i+2)!} - \frac{1}{(i+2)! (i+2)^k}\right). 
\end{align} 
\end{prop} 
\begin{proof} 
Let the coefficient terms, $c_j(i)$, be defined as in the next equation. 
\begin{equation} 
\label{eqn_cji_term_def_for_proofs} 
c_j(i) := \frac{(-1)^{j-i}}{i! (j-i)!} 
\end{equation} 
It follows from \eqref{eqn_S2Star_coeff_closed-form_sum_ident} that 
\begin{align*} 
\SIIStarCf{k+2}{j} & = \sum_{i=0}^{j-1} \left(c_j(j-i) - c_j(j+1-i) 
     \right) H_{j-i}^{(k)} \\ 
   & = 
     \sum_{i=0}^{j-1} \left(\frac{(-1)^i}{i! (j-i)!} + 
     \frac{(-1)^i i}{i! (j+1-i)!}\right) H_{j-i}^{(k)} \\ 
   & = 
     (j+1) \sum_{i=0}^{j-1} \frac{(-1)^i H_{j-i}^{(k)}}{i! (j+1-i)!}. 
\end{align*} 
The identities in the proposition statement follow by interchanging the 
summation indices in the last equation. 
\end{proof} 

\begin{prop}[Formulas Involving Real--Order Harmonic Numbers] 
\label{prop_MoreGen_Real-Order_HNumIdents_fro_S2StarCoeffs} 
For $k \in \mathbb{Z}^{+}$ and real $r \geq 0$, the 
generalized coefficients in \eqref{eqn_S2StarCfRecRelation} 
satisfy the following identities: 
\begin{align*} 
\SIIStarCf{k+2}{j} & = \sum_{i=0}^{j-1} 
     \frac{(-1)^{j-1-i} H_{i+1}^{(k-r)}}{(j-1-i)! (i+1)!}\left( 
     \frac{1}{(i+1)^r} + \frac{(j-1-i)}{(i+2)^{r+1}}\right) \\ 
\SIIStarCf{k+2}{j} & = \sum_{i=0}^{j-1} 
     \frac{(-1)^{j-1-i} H_{i+1}^{(k-r)}}{(j-1-i)! (i+1)!}\left( 
     \frac{1}{(i+1)^r} - \frac{1}{(i+2)^r} + \frac{(j+1)}{(i+2)^{r+1}}\right). 
\end{align*} 
\end{prop} 
\begin{proof} 
Let the coefficient terms, $c_j(i)$, be defined as in 
\eqref{eqn_cji_term_def_for_proofs}. 
It follows from \eqref{eqn_S2Star_coeff_closed-form_sum_ident} that 
\begin{align*} 
\SIIStarCf{k+2}{j} & = \sum_{i=0}^{j-1} \left( 
     \frac{c_j(j-i)}{(j-i)^r} - \frac{c_j(j+1-i)}{(j+1-i)^r} 
     \right) H_{j-i}^{(k-r)} \\ 
   & = 
     \sum_{i=0}^{j-1} \frac{(-1)^i H_{j-i}^{(k-r)}}{i! (j-i)!}\left( 
     \frac{1}{(j-i)^r} + \frac{i}{(j+1-i)^{r+1}}\right). 
\end{align*} 
The identities in the proposition follow similarly by interchanging the 
summation indices in the last equation. 
\end{proof} 

\subsection{Exponential Harmonic Number Sums} 

The \emph{Mathematica} \texttt{Sigma} package 
\citep{SYMB-SUM-COMB-SIGMAPKGDOCS} 
is able to obtain the formulas given in 
Corollary \ref{cor_Summary_S2Star_Hnr_HarmonicNumber_exps_k_in_2to6} by a 
straightforward procedure. 
The package is also able to verify the related results that 
\begin{align} 
\label{eqn_S2SCf_ExpHNum_spcase_idents_v1} 
\sum_{i=1}^j \frac{H_i}{i!} \frac{(-1)^{j-i}}{(j-i)!} & = 
     \frac{(-1)^{j-1}}{j\ j!} \\ 
\notag 
\sum_{i=1}^j \frac{H_i^{(2)}}{i!} \frac{(-1)^{j-i}}{(j-i)!} & = 
     \frac{(-1)^{j-1}}{j\ j!} H_j \\ 
\notag 
\sum_{i=1}^j \frac{H_i^{(3)}}{i!} \frac{(-1)^{j-i}}{(j-i)!} & = 
     \frac{(-1)^{j-1}}{2 j\ j!}\left(H_j^2 + H_j^{(2)}\right) \\ 
\notag 
\sum_{i=1}^j \frac{H_i^{(4)}}{i!} \frac{(-1)^{j-i}}{(j-i)!} & = 
     \frac{(-1)^{j-1}}{6 j\ j!}\left(H_j^3 + 3 H_j H_j^{(2)} + 
     2 H_j^{(3)}\right), 
\end{align} 
and then that 
\begin{align} 
\label{eqn_S2SCf_ExpHNum_spcase_idents_v2} 
\frac{H_j}{j!} & = \sum_{i=1}^j \frac{(-1)^{i-1}}{i\ i! (j-i)!} \\ 
\notag 
\frac{H_j^{(2)}}{j!} & = \sum_{i=1}^j \frac{(-1)^{i-1}}{i\ i! (j-i)!} H_i \\ 
\notag 
\frac{H_j^{(3)}}{j!} & = \sum_{i=1}^j \frac{(-1)^{i-1}}{2 i\ i! (j-i)!} 
     \left(H_i^2 + H_i^{(2)}\right) \\ 
\notag 
\frac{H_j^{(4)}}{j!} & = \sum_{i=1}^j \frac{(-1)^{i-1}}{6 i\ i! (j-i)!} 
     \left(H_i^3 + 3 H_i H_i^{(2)} + 2 H_i^{(3)}\right), 
\end{align} 
by considering the generating functions 
over each side of the equations in \eqref{eqn_S2SCf_ExpHNum_spcase_idents_v1}. 
An alternate, direct approach using the capabilities of 
\texttt{Sigma} for the special cases of the sums in 
\eqref{eqn_S2SCf_ExpHNum_spcase_idents_v2} is also used to obtain the 
closed--forms of these sums. 

Proposition \ref{prop_GenS2SCf_relations_to_ExpHNum_sums_gen_stmts-v1} 
provides a generalization of the coefficient sums given by these special cases. 
The particular expansions of the previous identities and their 
generalized forms in the proposition 
immediately imply relations between the exponential generating functions 
of the $r$--order harmonic numbers and of the generalized coefficients in 
\eqref{eqn_S2StarCfRecRelation}. 

\begin{prop} 
\label{prop_GenS2SCf_relations_to_ExpHNum_sums_gen_stmts-v1} 
For $k \in \mathbb{N}$ and $j \in \mathbb{Z}^{+}$, the 
generalized coefficients and exponential harmonic numbers are related 
through the following sums: 
\begin{align} 
\label{eqn_ExpHNum_Ident_for_GenS2StarCoeffs-v3} 
\SIIStarCf{k+2}{j} \cdot \frac{1}{j} & = \sum_{m=0}^{j} \frac{H_m^{(k+1)}}{m!} 
     \frac{(-1)^{j-m}}{(j-m)!} \\ 
\label{eqn_ExpHNum_Ident_for_GenS2StarCoeffs-v2} 
\frac{H_j^{(k+1)}}{j!} & = \sum_{i=1}^{j} \SIIStarCf{k+2}{i} \cdot
     \frac{1}{i(j-i)!}. 
\end{align} 
\end{prop} 
\begin{proof} 
We recall the series for the exponential harmonic numbers cited in 
\eqref{eqn_HNumEGFFromOGFIntTransform} of the introduction given by 
\begin{equation*}
\frac{H_n^{(k+1)}}{n!} = [z^n]\left( \sum_{j \geq 1} 
     \SIIStarCf{k+3}{j} z^j e^z\left( 1 + \frac{z}{j+1}\right) 
     \right).
\end{equation*}
Next, we see that the generating function, $\widehat{H}_k(z)$, of the 
\emph{$k$--order exponential harmonic numbers}, $H_n^{(k)} / n!$, is given by 
\begin{align*} 
\widehat{H}_{k+1}(z) e^{-z} & = \sum_{j \geq 1} 
     \SIIStarCf{k+3}{j} z^j \left(1 + \frac{z}{j+1}\right) \\ 
     & = 
     \sum_{j \geq 1} \SIIStarCf{k+3}{j} z^j + 
     \sum_{j \geq 2} \SIIStarCf{k+3}{j-1} \frac{z^j}{j} \\ 
     & = 
     \sum_{j \geq 1} \SIIStarCf{k+2}{j} \frac{z^j}{j}, 
\end{align*} 
by \eqref{eqn_S2StarCfRecRelation}, which implies the second identity in 
\eqref{eqn_ExpHNum_Ident_for_GenS2StarCoeffs-v2}. 
The proof of either \eqref{eqn_ExpHNum_Ident_for_GenS2StarCoeffs-v3} or 
\eqref{eqn_ExpHNum_Ident_for_GenS2StarCoeffs-v2} implies the result in the 
other equation by a formal power series, or generating function convolution, 
argument for establishing the forms of 
\eqref{eqn_S2SCf_ExpHNum_spcase_idents_v2} from the first set of results in 
\eqref{eqn_S2SCf_ExpHNum_spcase_idents_v1} (and vice versa). 
\end{proof} 

\begin{remark}[Functional Equations Resulting from the Binomial Transform] 
Notice that the results in \eqref{eqn_S2SCf_ExpHNum_spcase_idents_v2} imply 
new forms of functional equations between the \emph{polylogarithm functions}, 
$\Li_s(z) / (1-z)$, when $s = 2, 3$. 
For example, by integrating the generating function for the 
first--order harmonic numbers and applying the binomial transform, the 
second identity in the previous equations leads to the known functional 
equation for the \emph{dilogarithm function}, $\Li_2(z)$, providing that 
\citep[\S 25.12(i)]{NISTHB} \citep{ZAIGERDILOGFN}
\begin{equation*} 
\Li_2(z) = -\frac{1}{2} \Log(1-z)^2 - \Li_2\left(-\frac{z}{1-z}\right). 
\end{equation*} 
Similarly, the third identity in \eqref{eqn_S2SCf_ExpHNum_spcase_idents_v2} 
implies a new functional equation 
between products of the natural logarithm, the dilogarithm function, and the 
\emph{trilogarithm function}, $\Li_3(z)$, in the following form 
(\cf \emph{Landen's formula} for the trilogarithm): 
\begin{align*} 
\Li_3(z) & = -\frac{1}{6} \Log(1-z)^3 + 
     \frac{1}{2} \Log(1-z)^2 \Log\left(-\frac{z}{1-z}\right) - 
     \Log(1-z) \Li_2\left(\frac{1}{1-z}\right) \\ 
     & \phantom{=\ } - 
     \Log(1-z) \Li_2\left(-\frac{z}{1-z}\right) - 
     \Li_3\left(\frac{1}{1-z}\right) - 
     \Li_3\left(-\frac{z}{1-z}\right) - \zeta(3). 
\end{align*} 
\end{remark} 

\begin{remark}[Exponential Generating Functions for Harmonic Numbers] 
\label{remark_HNum_EGFs_first-order_closed-form_formula} 
The \textit{first--order harmonic numbers}, $H_n \equiv H_n^{(1)}$, 
have an explicit closed--form exponential generating function 
in $z$ given by \citep[\cf \S 5]{STIRESUMS} 
\begin{equation} 
\label{eqn_H1HatGF_SpCase_Closed-Form_GF_ident} 
\widehat{H}_1(z) := \sum_{n=0}^{\infty} H_n \frac{z^n}{n!} = 
     (-e^{z}) \sum_{k=1}^{\infty} \frac{(-z)^{k}}{k!\ k} \equiv 
     e^{z}\left(\gamma_{} + \Gamma(0, z) + \Log(z)\right). 
\end{equation} 
where $\gamma_{}$ denotes \textit{Euler's gamma constant} 
\citep[\S 5.2(ii)]{NISTHB} and 
$\Gamma(a, z)$ denotes the \textit{incomplete gamma function} 
\citep[\S 8]{NISTHB}. 
No apparent simple analogs to the closed--form function on the 
right--hand--side of \eqref{eqn_H1HatGF_SpCase_Closed-Form_GF_ident} are 
known for the exponential harmonic number generating functions, 
$\widehat{H}_k(z)$, when $k \geq 2$. 

However, we are able to easily relate these exponential generating functions 
to the generating functions of the sequence, 
$\SIIStarCf{k}{j}$, by applying 
Proposition \ref{prop_GenS2SCf_relations_to_ExpHNum_sums_gen_stmts-v1}. 
In particular, if we define the ordinary generating function of the 
sequences, $\SIIStarCf{k}{j}$, over $j \geq 1$ for fixed $k$ by 
$\widetilde{S}_{k,\ast}(z)$, the proposition immediately implies that 
(see Section \ref{subSection_NewAlmostLinear_HNumRecs}) 
\begin{align*}
\widehat{H}_{k+1}(z) & = 
     e^{z} \int_0^z \frac{\widetilde{S}_{k+2,\ast}(t)}{t} dt. 
\end{align*}
We compare these integral formulas to the somewhat simpler formal series 
expansion for the exponential harmonic numbers, 
$\widehat{H}_n^{(r)} = H_n^{(r)} / n!$, in the example from 
\eqref{eqn_HNumEGFFromOGFIntTransform} of the introduction in the form of 
\begin{equation*} 
\widehat{H}_{r}(z) = 
     \sum_{n \geq 1} \widehat{H}_n^{(r)} z^n = \sum_{j \geq 1} 
     \SIIStarCf{k+2}{j} \frac{z^j e^z}{(j+1)}\left(j+1+z\right).
\end{equation*}
Other relations between the exponential harmonic numbers, the generalized 
coefficients, $\SIIStarCf{k+2}{j}$, and the sequences, $M_{k+1}^{(d)}(z)$, 
are considered below in 
Section \ref{subSection_NewAlmostLinear_HNumRecs}. 
\end{remark} 

\subsection{New Recurrences and Expansions of the $k$--Order Harmonic Numbers 
            in Powers of $n$} 

\begin{remark}[Formulas for Integral Powers of $n$] 
\label{remark_Integer_Pows_of_n_S2_and_S2S_exp_stmts} 
For positive $n \in \mathbb{N}$, the following finite sums define the 
forms of the 
integral powers of $n$, given by $n^{k}$ and $n^{-k}$ 
for $k \in \mathbb{Z}^{+}$, 
respectively in terms of sums over the 
Stirling numbers of the second kind and the 
generalized transformation coefficients from \eqref{eqn_S2StarCfRecRelation}: 
\begin{align} 
\label{eqn_nPowk_coeff_closed-form_formula_v1} 
n^k & =  \sum_{j=1}^{k} \gkpSII{k}{j} \frac{n!}{(n-j)!} \\ 
\label{eqn_nPowk_coeff_closed-form_formula_v2} 
\frac{1}{n^{k}} & = \sum_{j=1}^{n} \SIIStarCf{k+2}{j} \frac{n!}{(n-j)!}. 
\end{align} 
A formula related to \eqref{eqn_nPowk_coeff_closed-form_formula_v1} 
cited in the references \citep[eq. (26.8.34); \S 26.8(v)]{NISTHB} 
is re--stated as follows for scalar--valued $x \neq 0, 1$: 
\begin{equation*} 
\sum_{j=0}^{n} j^k x^j = \sum_{j=0}^{k} 
     \gkpSII{k}{j} x^j \frac{d^{(j)}}{dx^{(j)}} \Biggl[ 
     \frac{1 - x^{n+1}}{1-x} 
     \Biggr]. 
\end{equation*} 
For fixed $k \in \mathbb{N}$ and $n \geq 0$, 
these partial sums 
can also be expressed in closed--form 
through the \textit{Bernoulli numbers}, $B_n$, defined as in 
\citep[\S ]{NISTHB} \citep[\S 6.5]{GKP} by 
\begin{equation} 
\notag 
S_k(n) := \sum_{j=0}^{n-1} j^k = \sum_{m=0}^{k} \binom{k+1}{m} 
     \frac{B_m n^{k+1-m}}{(k+1)}. 
\end{equation} 
For $k \in \mathbb{Z}^{+}$ and $n \in \mathbb{N}$, the integer--order 
harmonic number sequences, $H_n^{(k)}$, 
can then be defined recursively in terms of the 
generalized coefficient forms as follows 
when $n \geq 1$ and where $H_0^{(k)} \equiv 0$ for all $k \in \mathbb{Z}^{+}$: 
\begin{align} 
\notag 
H_n^{(k)} & = H_{n-1}^{(k)} + \frac{1}{n^k} \Iverson{n \geq 1} \\ 
\label{eqn_Hnk_IntegerOrder_HarmonicNumberSeq_Apps_S2StarSum_rdef} 
     & = 
     H_{n-1}^{(k)} + \sum_{j=1}^{n} \SIIStarCf{k+2}{j} 
     \frac{n!}{(n-j)!}. 
\end{align} 
We are now concerned with applying the results in the previous 
propositions to \eqref{eqn_nPowk_coeff_closed-form_formula_v2} 
in order to obtain new recurrence relations and sums for the $r$--order 
harmonic number sequences. 
\end{remark} 

\begin{cor}[Recurrences for the $k$--Order Harmonic Numbers] 
\label{cor_Other_Recurrences_for_k-Order-HNums_stmts_v3} 
Suppose that $k \in \mathbb{Z}^{+}$ and 
let $r \in [0, k) \subseteq \mathbb{R}$. 
The $k$--order harmonic numbers 
satisfy each of the following recurrence relations: 
\begin{align*} 
H_n^{(k)} & = H_{n-1}^{(k)} + \sum_{1 \leq i \leq j \leq n} 
     \binom{n}{j} \SIIStarCf{k+1}{i} (-1)^{j-i} (i-1)! \\ 
H_n^{(k)} & = 
     H_{n-1}^{(k)} + \sum_{1 \leq m \leq i \leq j \leq n} 
     \binom{n}{j} \binom{i}{m} (-1)^{j+m} H_m^{(k)} \\ 
H_n^{(k)} & = H_{n-1}^{(k)} + \sum_{0 \leq i < j \leq n} 
     \binom{n}{j} \binom{j}{i+1} (-1)^{j-1-i} H_{i+1}^{(k-r)} \left( 
     \frac{1}{(i+1)^{r}} - \frac{1}{(i+2)^{r}} + \frac{(j+1)}{(i+2)^{r+1}} 
     \right). 
\end{align*} 
\end{cor} 
\begin{proof} 
The first recurrence relation results by applying 
\eqref{eqn_Hnk_IntegerOrder_HarmonicNumberSeq_Apps_S2StarSum_rdef} to the 
identity 
\begin{equation*} 
\SIIStarCf{k+2}{j} = \frac{1}{j!} \sum_{1 \leq i \leq j} 
     \SIIStarCf{k+1}{i} (-1)^{j-i} (i-1)!. 
\end{equation*} 
The second and third identities are similarly obtained from 
\eqref{eqn_Hnk_IntegerOrder_HarmonicNumberSeq_Apps_S2StarSum_rdef} 
respectively combined with 
Proposition \ref{prop_GenS2SCf_relations_to_ExpHNum_sums_gen_stmts-v1} and 
Proposition \ref{prop_MoreGen_Real-Order_HNumIdents_fro_S2StarCoeffs}. 
\end{proof} 

\begin{prop}[Formulas for the $k$--Order Harmonic Numbers in Powers of $n$] 
\label{prop_Hnk_hnum_finite_sum_formulas-v2_S1_npow_exps} 
Let $k \in \mathbb{N}$ and $r \in [0, 1) \subseteq \mathbb{R}$.
For each $n \in \mathbb{N}$, the $k$--order harmonic numbers 
are given respectively through the next finite sums 
involving positive integer powers of $n+1$. 
\begin{align} 
\label{eqn_Hnk_hnum_finite_sum_formula-S1_npow_exps_v1} 
H_n^{(k)} & = \sum_{0 \leq j \leq n} \left(\sum_{0 \leq m \leq j+1} 
     \gkpSI{j+1}{m} \SIIStarCf{k+2}{j} \frac{(-1)^{j+1-m} (n+1)^{m}}{(j+1)} 
     \right) 
\end{align} 
\end{prop} 
\begin{proof} 
For $n, j \in \mathbb{N}$ 
we have the next expansion of the binomial coefficients as a 
finite sum over powers of $n + 1$ given in 
\eqref{eqn_proof_binom_coeff_S1NPow_exp-tagged_v.i}. 
\begin{align} 
\notag 
\binom{n+1}{j+1} & = \frac{(n+1)!}{(j+1)! \cdot (n+1-(j+1))!} \\ 
\tag{i} 
\label{eqn_proof_binom_coeff_S1NPow_exp-tagged_v.i} 
   & = 
     \sum_{m=0}^{j+1} \gkpSI{j+1}{m} \frac{(-1)^{j+1-m} (n+1)^{m}}{(j+1)!}. 
\end{align} 
The result in \eqref{eqn_Hnk_hnum_finite_sum_formula-S1_npow_exps_v1} 
then follows from the identity 
\begin{align*} 
H_n^{(k)} & = \sum_{0 \leq j \leq n} \binom{n+1}{j+1} \SIIStarCf{k+2}{j} j!.
     \qedhere
\end{align*}
\end{proof} 

\section{Applications} 
\label{Section_Applications} 

\subsection{New Series for Polylogarithm--Related Functions and Special Values} 

For the geometric series special case of the transform in 
\eqref{eqn_S2StarCoeffGFTransform} with 
$g_n \equiv 1$ for all $n$, the next transformation 
stated in \eqref{eqn_LikFn_S2Star_gen_series_exp_def} is employed to expand the 
polylogarithm function, $\Li_s(z) = \sum_{n \geq 1} z^n / n^s$, 
in terms of only the $j^{th}$ derivatives, $G^{(j)}(z) := j! / (1-z)^{j+1}$, 
as a series analog to \eqref{eqn_Lisz_ModHZetaFn_KnownSeriesExps} given by 
\begin{align} 
\label{eqn_LikFn_S2Star_gen_series_exp_def} 
\Li_{s}(z) & = 
     \sum_{j=1}^{\infty} \SIIStarCf{s+2}{j} \frac{z^j j!}{(1-z)^{j+1}}. 
\end{align} 

\begin{cor}[Polylogarithm Functions] 
\label{cor_NewSeries_for_PolyLogFns_SpCases_s1234} 
The polylogarithm functions, $\Li_k(z)$, for 
$k \in [1, 4] \subseteq \mathbb{N}$ are expanded as the 
following special case series: 
\begin{align*} 
\Li_1(z) & = \sum_{j=1}^{\infty} (-1)^{j-1} H_j \frac{z^j}{(1-z)^{j+1}} \\ 
\Li_2(z) & = \sum_{j=1}^{\infty} \frac{(-1)^{j-1}}{2} \left( 
     H_j^2 + H_j^{(2)}\right) \frac{z^j}{(1-z)^{j+1}} \\ 
\Li_3(z) & = \sum_{j=1}^{\infty} \frac{(-1)^{j-1}}{6} \left( 
     H_j^3 + 3 H_j H_j^{(2)} + 2 H_j^{(3)}\right) \frac{z^j}{(1-z)^{j+1}} \\ 
\Li_4(z) & = \sum_{j=1}^{\infty} \frac{(-1)^{j-1}}{24} \left( 
     H_j^4 + 6 H_j^2 H_j^{(2)} + 3 \left(H_j^{(2)}\right)^2 + 
     8 H_j H_j^{(3)} + 6 H_j^{(4)}\right) \frac{z^j}{(1-z)^{j+1}}. 
\end{align*} 
\end{cor} 
\begin{proof} 
These series follow from the coefficient identities given in 
\eqref{eqn_S2Star_Hnr_HarmonicNumber_exps_k_in_2to6} 
of Corollary \ref{cor_Summary_S2Star_Hnr_HarmonicNumber_exps_k_in_2to6} 
applied to the transformed series of the polylogarithm function 
in \eqref{eqn_LikFn_S2Star_gen_series_exp_def}. 
\end{proof} 

\begin{example}[The Alternating Zeta Function] 
\label{example_AltZetaFn_def_PolyLogFn_SpCaseSeries_exps_stmts_v1} 
Let $s \in \mathbb{Z}^{+}$ and consider the following forms of the 
\textit{alternating zeta function}, 
$\zeta^{\ast}(s) = \Li_s(-1)$, 
defined as in the references \citep[\S 7]{FLAJOLETESUMS} \footnote{ 
     The alternating zeta function, $\zeta^{\ast}(s)$, is defined 
     in the alternate notation of $\bar{\zeta}(s)$ for the function 
     in the reference \citep[\S 7]{FLAJOLETESUMS}.      
     The series for $\zeta^{\ast}(s)$ is also 
     commonly denoted by the \textit{Dirichlet eta function}, 
     $\eta(s)$, also as noted in the reference 
     \citep[eq. (3); \S 2]{DOUBLEINTS-LTRANS}. 
}. 
\begin{align} 
\label{eqn_example_AltZetaFn_def_PolyLogFn_SpCaseSeries_ZetaStar_prop_v1} 
\zeta^{\ast}(s) & := \sum_{n=1}^{\infty} \frac{(-1)^{n-1}}{n^s} = 
     \left(1-2^{1-s}\right) \zeta(s) \cdot \Iverson{s > 1} + 
     \Log(2) \cdot \Iverson{s = 1} 
\end{align} 
Since $\zeta^{\ast}(s) \equiv -\Li_s(-1)$, 
the transformed series for the polylogarithm function in 
\eqref{eqn_LikFn_S2Star_gen_series_exp_def} 
leads to series expansions given by 
\begin{align} 
\label{eqn_example_AltZetaFn_def_PolyLogFn_SpCaseSeries_ZetaStar_prop_v2} 
\zeta^{\ast}(s) & \phantom{:} = 
     \sum_{j=1}^{\infty} \SIIStarCf{s+2}{j} 
     \frac{(-1)^{j-1} \cdot j!}{2^{j+1}}. 
\end{align} 
The coefficient formulas in 
Corollary \ref{cor_Summary_S2Star_Hnr_HarmonicNumber_exps_k_in_2to6} 
are then applied to obtain the 
following new series results for the next few 
special cases of the alternating zeta function constants in 
\eqref{eqn_example_AltZetaFn_def_PolyLogFn_SpCaseSeries_ZetaStar_prop_v1}
\footnote{
     These formulas are compared to the Bell polynomial expansions of the 
     identity in \eqref{eqn_S2StarCf_HNum_Heuristic_formula_stmt} 
     cited in \citep[\S 3]{CONNON-BELL-POLY-APPS}
}: 
\begin{align*} 
\zeta^{\ast}(1) & = \sum_{j=1}^{\infty} \frac{H_j}{2 \cdot 2^{j}} \equiv 
     \Log(2) \\ 
\zeta^{\ast}(2) & = \sum_{j=1}^{\infty} 
     \frac{\left(H_j^2 + H_j^{(2)}\right)}{4 \cdot 2^{j}} \equiv 
     \frac{1}{2} \cdot \frac{\pi^2}{6} \\ 
\zeta^{\ast}(3) & = \sum_{j=1}^{\infty} 
     \frac{\left(H_j^3 + 3 H_j H_j^{(2)} + 2 H_j^{(3)}\right)}{12 \cdot 2^{j}} 
     \equiv \frac{3}{4} \cdot \zeta(3) \\ 
\zeta^{\ast}(4) & =  \sum_{j=1}^{\infty} 
     \frac{\left(H_j^4 + 6 H_j^2 H_j^{(2)} + 3 \left(H_j^{(2)}\right)^2 + 
     8 H_j H_j^{(3)} + 6 H_j^{(4)}\right)}{48 \cdot 2^{j}} \equiv 
     \frac{7}{8} \cdot \frac{\pi^4}{90}. 
\end{align*} 
Notice that since the exponential generating function for the 
Stirling numbers of the first kind, $\gkpSI{j+1}{k+1} / j!$, is given by 
\begin{equation*} 
\sum_{j \geq 0} \gkpSI{j+1}{k+1} \frac{z^j}{j!} = 
     \frac{(-1)^k}{k! \cdot (1-z)} \Log\left(\frac{1}{1-z}\right)^k, 
\end{equation*} 
we may also write the left--hand--side sums in the previous equations in 
terms of powers of $\Log(2)$ and partial Euler--like sums involving 
weighted terms of harmonic numbers. 
For example, the series for $\zeta^{\ast}(s)$ for $3 \leq s \leq 5$ 
are expanded as 
\citep[\S 2]{STIRESUMS} \citep{GKP} 
\begin{align*} 
\zeta^{\ast}(3) & = 
     \frac{1}{6} \Log(2)^3 + \sum_{j \geq 0} \frac{H_j H_j^{(2)}}{2^{j+1}} 
     && \approx 0.901543 \\ 
\zeta^{\ast}(4) & = \frac{1}{24} \Log(2)^4 + 
     \sum_{j \geq 0} \frac{H_j^2 H_j^{(2)}}{2^{j+2}} + 
     \sum_{j \geq 0} \frac{H_j H_j^{(3)}}{2^{j+2}}  && \approx 0.947033 \\ 
\zeta^{\ast}(5) & = \frac{1}{120} \Log(2)^5 + 
     \sum_{j \geq 0} \frac{H_j^3 H_j^{(2)}}{12 \cdot 2^j} + 
     \sum_{j \geq 0} \frac{H_j^{(2)} H_j^{(3)}}{6 \cdot 2^j} + 
     \sum_{j \geq 0} \frac{H_j H_j^{(4)}}{2^{j+2}}
     && \approx 0.972120. 
\end{align*}
Other identities for the partial sums of the right--hand--side sums in the 
previous equations are obtained through the \texttt{Sigma} package. 
\end{example} 

\begin{example}[Fourier Series for the Periodic Bernoulli Polynomials] 
The \emph{Bernoulli polynomials}, $B_n(x)$, have the 
exponential generating function \citep[\S 24.2]{NISTHB}
\begin{align*} 
\sum_{n \geq 0} \frac{B_n(x)}{n!} z^n & = \frac{z e^{xz}}{e^z-1}. 
\end{align*} 
These polynomials also satisfy Fourier series of the following forms 
when $k \geq 0$ and where $\{x\}$ denotes the \emph{fractional part} of 
$x \in \mathbb{R}$ \citep[\cf \S 24.8(i)]{NISTHB}: 
\begin{align} 
\notag 
\frac{B_{2k+2}\left(\{x\}\right)}{(2k+2)!} & = 
     \frac{2 (-1)^{k+1}}{(2\pi)^{2k+2}} \times 
     \sum_{n \geq 0} \frac{(-1)^n}{n^{2k+2}} 
     \cos\left[2\pi n\left(x - \frac{1}{2}\right)\right] \\ 
\notag 
     & = 
     \frac{(-1)^{k+1}}{(2\pi)^{2k+2}} \sum_{j \geq 1} \SIIStarCf{2k+4}{j} \left[ 
     \frac{\left(e^{2\pi\imath(x-1/2)}\right)^j}{ 
     \left(1+e^{2\pi\imath(x-1/2)}\right)^{j+1}} + 
     \frac{\left(e^{-2\pi\imath(x-1/2)}\right)^j}{ 
     \left(1+e^{-2\pi\imath(x-1/2)}\right)^{j+1}} 
     \right] \\ 
\notag 
\frac{B_{2k+1}\left(\{x\}\right)}{(2k+1)!} & = 
     \frac{2 (-1)^k}{(2\pi)^{2k+1}} \times 
     \sum_{n \geq 0} \frac{(-1)^n}{n^{2k+1}} 
     \sin\left[2\pi n\left(x - \frac{1}{2}\right)\right] \\ 
\notag 
     & = 
     \frac{(-1)^k}{(2\pi)^{2k+1} \cdot \imath} 
     \sum_{j \geq 1} \SIIStarCf{2k+3}{j} \left[ 
     \frac{\left(e^{2\pi\imath(x-1/2)}\right)^j}{ 
     \left(1+e^{2\pi\imath(x-1/2)}\right)^{j+1}} -
     \frac{\left(e^{-2\pi\imath(x-1/2)}\right)^j}{ 
     \left(1+e^{-2\pi\imath(x-1/2)}\right)^{j+1}} 
     \right] 
\end{align} 
Several particular examples of the series for the 
\emph{periodic Bernoulli polynomial} 
variants, $B_k(\{x\}) / k!$, expanded by the previous equations are given by 
\begin{align*} 
B_1\left(\left\{5/4\right\}\right) & = 
     \left\{\frac{5}{4}\right\} - \frac{1}{2} \\ 
     & = 
     \frac{(\imath+1)}{4\pi\imath} \times \sum_{j \geq 0} \frac{H_j}{2^j} \left( 
     (1-\imath)^j + \imath (1 + \imath)^j 
     \right) \\ 
\frac{B_2\left(\left\{\frac{5}{4}\right\}\right)}{2} & = 
     \left\{\frac{5}{4}\right\}^2 -
     \left\{\frac{5}{4}\right\} - \frac{1}{6} \\ 
     & = 
     -\frac{(\imath+1)}{16 \pi^2} \times \sum_{j \geq 0} 
     \frac{H_j^2+H_j^{(2)}}{2^j} \left( 
     (1-\imath)^j - \imath (1 + \imath)^j 
     \right) \\ 
\frac{B_3\left(\left\{\frac{11}{4}\right\}\right)}{6} & = 
     \left\{\frac{11}{4}\right\}^3 - \frac{3}{2} \left\{\frac{11}{4}\right\}^2 + 
     \frac{1}{2} \left\{\frac{11}{4}\right\} \\ 
     & = 
     -\frac{(\imath+1)}{48 \pi^2 \cdot \imath} \times \sum_{j \geq 0} 
     \left(H_j^3+3 H_j H_j^{(2)}+2 H_j^{(3)}\right)
     \frac{\left(\imath^j - \imath\right)}{\left(1+\imath\right)^{j+1}}. 
\end{align*} 
More generally, for $k = 1, 2$ and any $x \in \mathbb{R}$ we may write the 
periodic Bernoulli polynomials in the forms of 
\begin{align*} 
B_1\left(\left\{x\right\}\right) & = 
     \frac{1}{2\pi\imath} 
     \Log\left(\frac{1-e^{2\pi\imath \cdot x}}{1-e^{-2\pi\imath \cdot x}} 
     \right) \\ 
\frac{B_2\left(\left\{x\right\}\right)}{2} & = 
     -\frac{1}{4 \pi^2} \sum_{b = \pm 1} \left( 
     \Log(1-e^{2\pi\imath \cdot b x})^2 + 2 
     \Li_2\left(\frac{1}{2}(1+b \imath \cot(\pi x))\right) 
     \right). 
\end{align*}
\end{example} 

\subsection{Almost Linear Recurrence Relations for the 
            $k$--Order Harmonic Numbers} 
\label{subSection_NewAlmostLinear_HNumRecs}

We first define the sequences, $M_{k+1}^{(d)}(n)$, for integers 
$k > 2$, $d \geq 1$, and $n \geq 0$ as 
\begin{equation} 
\label{eqn_Mkdn_HNumSum_def} 
M_{k+1}^{(d)}(n) = \sum_{1 \leq m \leq d} \gkpSI{d}{m} H_n^{(k+1-m)}.
\end{equation} 
We have an alternate sum for these terms proved in the following 
proposition. 

\begin{prop}[An Alternate Sum Identity] 
\label{prop_AltSumIdent_for_Mkdn}
For integers  $k > 2$, $d \geq 1$, and $n \geq 0$, we have that the 
harmonic number sums, $M_{k+1}^{(d)}(n)$, in 
\eqref{eqn_Mkdn_HNumSum_def} satisfy 
\begin{equation} 
\label{eqn_Mkdn_AltSum_exp} 
M_{k+1}^{(d)}(n) = \sum_{1 \leq j \leq n} 
     \binom{n}{j} \SIIStarCf{k+2}{j} \frac{(-1)^{j}}{(j+d)} 
     \cdot \frac{(n+d)!}{n!}.
\end{equation} 
\end{prop}
\begin{proof} 
We first use the \emph{RISC} \emph{Mathematica} 
package \texttt{Guess}\footnote{ 
     \url{https://www.risc.jku.at/research/combinat/software/ergosum/RISC/Guess.html}.
} 
to obtain a short proof that both \eqref{eqn_Mkdn_HNumSum_def} and 
\eqref{eqn_Mkdn_AltSum_exp} 
satisfy the same homogeneous recurrence relation given by 
\begin{align*} 
M_{k+1}^{(d)}(n) - M_{k+1}^{(d)}(n + 1) + (n+2) M_{k+2}^{(d)}(n + 1) - 
     (n+2) M_{k+2}^{(d)}(n) & = 0, 
\end{align*} 
though many other variants of this recurrence are formulated similarly. 
We then deduce from this observation that the two formulas are equivalent 
representations of the harmonic number sums in 
\eqref{eqn_Mkdn_HNumSum_def}. 
\end{proof}

We use the definition in \eqref{eqn_Mkdn_HNumSum_def} for 
multiple cases of $d \in \mathbb{Z}^{+}$ to obtain new almost linear 
recurrence relations between the $r$--order harmonic numbers over $r$ 
with ``\textit{remainder}'' terms given in terms of the sums in 
\eqref{eqn_Mkdn_AltSum_exp} of 
Proposition \ref{prop_AltSumIdent_for_Mkdn}.
The next corollary provides several particular examples. 

\begin{cor}[New Almost Linear Recurrence Relations for the Harmonic Numbers]
\label{cor_NewAlmostLinear_recs_for_HNums} 
For $n \geq 0$, $k \in \mathbb{Z}^{+}$, and any fixed $m \in \mathbb{R}$, 
we have the following ``\textit{almost linear}'' recurrence relations for the 
harmonic number sequences involving so--termed ``\textit{remainder}'' terms 
given by the sequences defined in \eqref{eqn_Mkdn_AltSum_exp}: 
\begin{align*} 
H_n^{(k)} & = H_n^{(k-2)} - 3 M_{k+1}^{(2)}(n) + M_{k+1}^{(3)}(n) \\ 
2 H_n^{(k)} & = -3 H_n^{(k-1)} - H_n^{(k-2)} - M_{k+1}^{(3)}(n) \\ 
7 H_n^{(k)} & = -12 H_n^{(k-1)} + 6 H_n^{(k-2)} - H_n^{(p-3)} - 
     M_{k+1}^{(2)}(n) + M_{k+1}^{(4)}(n) \\ 
5 H_n^{(k)} & = -9 H_n^{(k-1)} - 5 H_n^{(k-2)} - H_n^{(p-3)} - 
     M_{k+1}^{(2)}(n) + M_{k+1}^{(3)}(n) - M_{k+1}^{(4)}(n) \\ 
H_n^{(k)} & = 2 H_n^{(k-2)} - H_n^{(p-3)} + 
     M_{k+1}^{(2)}(n) - 4 M_{k+1}^{(3)}(n) + M_{k+1}^{(4)}(n) \\ 
H_n^{(k)} & = (1-m) H_n^{(k-2)} + m H_n^{(p-4)} - 
     (12m+3) M_{k+1}^{(2)}(n) + (24m+1) M_{k+1}^{(3)}(n) \\ 
     & \phantom{=\ } - 
     10m M_{k+1}^{(4)}(n) + m M_{k+1}^{(5)}(n). 
\end{align*}
\end{cor}
\begin{proof}
We are able to prove these recurrences as special cases of more general 
equations obtained from \eqref{eqn_Mkdn_HNumSum_def} by Gaussian elimination. 
Specifically, for constants $a_i,b_i,c_i \in \mathbb{R}$ and $d \neq 0$, the 
following more general recurrence relations follow from 
\eqref{eqn_Mkdn_HNumSum_def}: 
\begin{align*}
d H_n^{(k)} & = 
     a_1 H_n^{(k-1)} + \left(a_1+d\right) H_n^{(k-2)} + 
     \left(2 a_1 + 3 d\right) M_{k+1}^{(2)}(n) + 
     \left(a_1 + d\right) M_{k+1}^{(3)}(n) \\ 
     & = 
     b_1 H_n^{(k-1)} + b_2 H_n^{(k-2)} - \left(b_1-b_2+d\right) H_n^{(k-3)} - 
     \left(6 b_1 - 4 b_2 + 7 d\right) M_{k+1}^{(2)}(n) \\ 
     & \phantom{=\ } + 
     \left(6 b_1 - 5 b_2 + 6 d\right) M_{k+1}^{(3)}(n) - 
     \left(b_1 - b_2 + d\right) M_{k+1}^{(4)}(n) \\ 
     & = 
     c_1 H_n^{(k-1)} + c_2 H_n^{(k-2)} + c_3 H_n^{(k-3)} +
     \left(c_1-c_2+c_3+d\right) H_n^{(k-4)} \\ 
     & \phantom{=\ } - 
     \left(14 c_1 - 12 c_2 + 8 c_3 + 15 d\right) M_{k+1}^{(2)}(n) + 
     \left(25 c_1 - 24 c_2 + 19 c_3 + 25 d\right) M_{k+1}^{(3)}(n) \\ 
     & \phantom{=\ } - 
     \left(10 c_1 - 10 c_2 + 9 c_3 + 10 d\right) M_{k+1}^{(4)}(n) + 
     \left(c_1 - c_2 + c_3 + d\right) M_{k+1}^{(5)}(n). 
\end{align*} 
Each of the recurrences listed in the corollaries above follow as special 
cases of these results.
\end{proof}

\begin{remark}[Applications of the Corollary]
The limiting behavior of the recurrences given in the corollary 
provide new almost linear relations between the integer--order 
zeta constants and remainder terms expanded in the form of 
\eqref{eqn_Mkdn_AltSum_exp}. 
The new recurrences given in 
Corollary \ref{cor_NewAlmostLinear_recs_for_HNums} 
suggest an inductive approach to the limiting behaviors of these 
harmonic number sequences, and to the properties of the 
zeta function constants, $\zeta(2k+1)$, for integers $k \geq 1$. 
The catch with this approach is finding non--trivial approximations and 
limiting behaviors for the remainder terms, $M_{k+1}^{(d)}(n)$. 

While the zeta function constants, $\zeta(2k)$ for $k \geq 1$, are known in 
closed--form through the \emph{Bernoulli numbers} and rational 
multiples of powers of $\pi$, comparatively little is 
known about properties of the odd--indexed zeta constants, $\zeta(2k+1)$, 
when $k \geq 1$, with the exception of \emph{Ap\'{e}ry's constant}, $\zeta(3)$, 
which is known to be irrational. 
We do however know that infinitely--many of these constants are irrational, and 
that at least one of the constants, 
$\zeta(5)$, $\zeta(7)$, $\zeta(9)$, or $\zeta(11)$, must be irrational 
\citep{ONEOFTHESE-ZETAFN}. 
Statements of recurrence relations between the zeta functions 
over the positive integers of this type are apparently new, and offer 
new inductively--phrased insights to the properties of these constants 
as considered in the special case example below. 
\end{remark} 

\begin{example}[Generating Functions for the Remainder Term Sequences]
\label{example_Mkdn_recs_and_EGFs} 
We can obtain the following coefficient forms of the exponential and 
ordinary generating functions for the remainder terms, $M_{k+1}^{(d)}(n)$, 
both directly from \eqref{eqn_Mkdn_AltSum_exp} and by 
applying the the binomial transform to the 
corresponding generating functions for the sequences, $\SIIStarCf{k}{j}$, 
respectively denoted by $\widetilde{S}_{k,\ast}(z)$ and 
$\widehat{S}_{k,\ast}(z)$: 
\begin{align*} 
\frac{M_{k+1}^{(d)}(n)}{n!} & = [z^n]\left( \sum_{0 \leq i \leq d} 
     \binom{d}{i}^2 (d-i)! 
     z^i D_z^{(i)}\left[\frac{e^z}{(-z)^d} 
     \int_0^{-z} t^{d-1} \widehat{S}_{k+2,\ast}(t) dt 
     \right] \right) \\ 
M_{k+1}^{(d)}(n) & = \frac{(n+d)!}{n!} \cdot [z^n]\left(
     \frac{e^{z/(1-z)}}{1-z} \int_0^{-z/(1-z)} 
     t^{d-1} \widetilde{S}_{k+2,\ast}(t) dt\right) \\ 
     & = 
     [z^n]\left( \sum_{0 \leq i \leq d} 
     \binom{d}{i}^2 (d-i)! 
     z^i D_z^{(i)}\left[\frac{e^{z/(1-z)}}{1-z} 
     \left(\frac{-z}{1-z}\right)^{-d} \int_0^{-\frac{z}{1-z}} 
     t^{d-1} \widetilde{S}_{k+2,\ast}(t) dt \right] \right). 
\end{align*} 
The forms of these generating functions imply relations between the 
exponential series functions, 
$\widehat{\Li}_s(z) = \sum_{n \geq 1} \frac{z^n}{n^s n!}$, and the 
exponential harmonic number generating functions, as well as between the 
zeta function constants, $\zeta(k)$, for integers $k \geq 2$ by 
Corollary \ref{cor_NewAlmostLinear_recs_for_HNums}. 
For example, we may relate the first two cases of the zeta function 
constants, $\zeta(2k+1)$, over the odd positive integers by 
asymptotically estimating the limiting behavior of the sums 
$M_6^{(2)}(n)$ and $M_6^{(3)}(n)$ as\footnote{ 
     We also note that the following formula 
     is obtained from the series in 
     \eqref{eqn_Lisz_ModHZetaFn_KnownSeriesExps} by a reverse binomial 
     transform operation for $k \geq 0$ and $j \geq 1$: 
     \begin{equation*} 
     \SIIStarCf{k+2}{j} = \frac{(-1)^j}{(j-1)!} [z^j] 
          \Li_{k+1}\left(-\frac{z}{1-z}\right). 
     \end{equation*} 
} 
\begin{align*} 
\zeta(5) & = \zeta(3) - \lim_{n \rightarrow \infty}\left( 
     3 M_6^{(2)}(n) - M_6^{(3)}(n) \right) \\ 
     & = 
     \zeta(3) + \lim_{n \rightarrow \infty} [z^n]\Biggl( 
     \frac{(3z-1)}{(1-z)^4} 
     \widetilde{S}_{8,\ast}\left(-\frac{z}{1-z}\right) - 
     \frac{z (3z+1)}{(1-z)^5} 
     \widetilde{S}^{\prime}_{8,\ast}\left(-\frac{z}{1-z}\right) \\ 
     & \phantom{\qquad\qquad} + 
     \frac{z^2}{(1-z)^6} 
     \widetilde{S}^{\prime\prime}_{8,\ast}\left(-\frac{z}{1-z}\right) 
     \Biggl). 
\end{align*}
Other special case relations between zeta function constants are 
constructed similarly from 
Corollary \ref{cor_NewAlmostLinear_recs_for_HNums}. 
\end{example} 

\section{Conclusions} 
\label{Section_Conclusions} 

\subsection*{Summary of Results} 

The generalized coefficients implicitly defined through the transformation 
result in \eqref{eqn_S2StarCoeffGFTransform} 
satisfy a number of properties and relations analogous to those of the 
Stirling numbers of the second kind. The form of these implicit 
transformation coefficients satisfy a non--triangular, two--index 
recurrence relation given in \eqref{eqn_S2StarCfRecRelation} 
that can effectively be viewed as the 
Stirling number recurrence from \eqref{eqn_S2_rdef} 
``\textit{in reverse}''. 
The coefficients defined recursively 
through \eqref{eqn_S2StarCfRecRelation} can 
alternately be viewed as a generalization of the 
Stirling numbers of the second kind 
in the context of \eqref{eqn_S2NPowCoeffGFTransform} 
for $k \in \mathbb{Z} \setminus \mathbb{N}$. 

There are a plethora of additional harmonic number identities and 
recurrence relations that are derived from the identities given in 
Section \ref{Section_InitProps_OGFs_S1Relations_of_GenCoeffs}. 
We may also use the binomial transform with the new polylogarithm function 
series in \eqref{eqn_LikFn_S2Star_gen_series_exp_def} 
to give new proofs of well--known functional equations satisfied by the 
dilogarithm and trilogarithm functions. 
Since the truncated polylogarithm series and ordinary harmonic number 
generating functions are always rational, we may adapt these 
generalized series expansions to enumerate \textit{Euler--sum--like} 
series with weighted coefficients of the form 
$H_n^{(\pi_1)} \cdots H_n^{(\pi_k)} / n^s$ \citep[\S 6.3]{ECV2}. 

\subsection*{Generalizations} 

The interpretation of the transformation coefficients in
\eqref{eqn_S2StarCfRecRelation} as the finite sum in 
\eqref{eqn_S2Star_coeff_closed-form_sum_ident} 
motivates several generalizations briefly outlined below. 
For example, given any non--zero sequence, $\langle f(n) \rangle$, we 
may define a formal series transformation for the corresponding sums 
\begin{equation*} 
\sum_{n \geq 1} \frac{g_n}{f(n)^k} z^n = \sum_{j \geq 1} 
     \gkpSII{k+2}{j}_{f,\ast} z^j G^{(j)}(z), 
\end{equation*} 
where the modified coefficients, $\gkpSII{k+2}{j}_{f,\ast}$, 
are given by the sum 
\begin{equation*}
\gkpSII{k+2}{j}_{f,\ast} = \frac{1}{j!} \sum_{m=1}^{j} 
     \binom{j}{m} \frac{(-1)^{j-m}}{f(m)^{k}}. 
\end{equation*}
When $f(n) = \alpha n + \beta$, we can derive a number of similar identities 
to the relations established in this article in terms of the partial sums 
of the modified Hurwitz zeta function, $\Phi(z, s, \alpha, \beta)$ 
\citep[\cf \S 3]{FS-NORLUND-RICE-INTS}. 

Suppose that $p/q \in \mathbb{Q}^{+}$ and let the 
rational--order series transformation with respect to $z$ be defined as 
\begin{equation} 
\notag 
\QT_{p/q}\left[F(z)\right] := \sum_{n=0}^{\infty} n^{p/q} \cdot 
     [z^n] F(z) \cdot z^n. 
\end{equation} 
One topic for further exploration is generalizing the first transformation 
involving the Stirling numbers of the second kind in 
\eqref{eqn_S2NPowCoeffGFTransform} to analogous finite sum expansions that 
generate the positive rational--order series in the previous equation. 
If $|z| < 1$, the function $\Li_{s+1}(z)$ is given by 
\begin{equation} 
\notag 
\Li_{s+1}(z) = \frac{z \cdot (-1)^{s}}{s!} \int_0^1 
     \frac{\Log(t)^{s}}{(1-tz)}\ dt, 
\end{equation} 
which is evaluated termwise with respect to $z$ as 
\cite[eq. (4); \S 2]{EXPLICIT-EVAL-ESUMS} 
\begin{equation} 
\notag 
\frac{1}{n^{s+1}} = \frac{(-1)^{s}}{s!} \int_0^1 
     t^{n-1} \cdot \Log(t)^{s}\ dt. 
\end{equation} 
Other possible approaches to formulating the transformations of these series 
include considering series involving \textit{fractional derivatives} 
described briefly in the references \citep[\S 1.15(vi)--(vii)]{NISTHB}. 
Another alternate approach is to consider a shifted series in 
powers $(n \pm 1)^{p/q}$ 
that then employs an expansion over non--negative integral powers of $n$ 
with coefficients in terms of binomial coefficients, 
though the resulting transformations in this case are no longer 
formulated as finite sums as in the formula from 
\eqref{eqn_S2NPowCoeffGFTransform} 
when the exponent of $p/q$ assumes values over the 
non--integer, positive rational numbers. 


\vskip 0.1in 
\textbf{\hrule}\bigskip 
\vskip 0.1in 

\renewcommand{\refname}{References} 
\bibliographystyle{abbrv-mod} 

\begin{thebibliography}{10}

\bibitem{STIRESUMS}
V.~Adamchik, On {S}tirling numbers and {E}uler sums, {\em J. Comput. Appl.
  Math.} {\bf 79} (1997),  119--130.

\bibitem{EXPLICIT-EVAL-ESUMS}
D.~Borwein, J.~M. Borwein, and R.~Girgensohn, Explicit evaluation of {E}uler
  sums,   (1994).

\bibitem{CONNON-BELL-POLY-APPS}
D.~F. Connon, Various applications of the (exponential) {B}ell polynomials,
  {\em arXiv:Math.CA/1001.2835}  (2010).

\bibitem{FLAJOLETESUMS}
P.~Flajolet and B.~Salvy, {E}uler sums and contour integral representations,
  {\em Experimental Mathematics} {\bf 7} (1998).

\bibitem{FS-NORLUND-RICE-INTS}
P.~Flajolet and R.~Sedgewick, {M}ellin transforms and asymptotics: Finite
  differences and {R}ice's integral, {\em Theoretical Computer Science}
  (1995),  101--124.

\bibitem{GKP}
R.~L. Graham, D.~E. Knuth, and O.~Patashnik, {\em Concrete Mathematics: A
  Foundation for Computer Science}, Addison-Wesley, 1994.

\bibitem{DOUBLEINTS-LTRANS}
J.~Guillera and J.~Sondow, Double integrals and infinite products for some
  classical constants via analytic continuations of {L}erch's transcendent,
  (2006), ~21.

\bibitem{TAOCPV1}
D.~E. Knuth, {\em The Art of Computer Programming: Fundamental Algorithms},
  Vol.~1, Addison-Wesley, 1997.

\bibitem{NISTHB}
F.~W.~J. Olver, D.~W. Lozier, R.~F. Boisvert, and C.~W. Clark, eds., {\em
  {NIST} Handbook of Mathematical Functions}, Cambridge University Press, 2010.

\bibitem{MULTIFACTJIS}
M.~D. Schmidt, Generalized $j$--factorial functions, polynomials, and
  applications, {\em J. Integer Seq.} {\bf 13} (2010).

\bibitem{SYMB-SUM-COMB-SIGMAPKGDOCS}
C.~Schneider, Symbolic summation assists combinatorics, {\em Sem.~Lothar.
  Combin.} {\bf 56} (2007),  1--36.

\bibitem{ECV2}
R.~P. Stanley, {\em Enumerative Combinatorics}, Vol.~2, Cambridge, 1999.

\bibitem{GFOLOGY}
H.~S. Wilf, {\em Generatingfunctionology}, Academic Press, 1994.

\bibitem{ZAIGERDILOGFN}
D.~Zaiger, The dilogarithm function, {\em Frontiers in number theory, physics,
  and geometry II}  (2007),  3--65.

\bibitem{ONEOFTHESE-ZETAFN}
W.~Zudilin, One of the numbers $\zeta(5)$, $\zeta(7)$, $\zeta(9)$, $\zeta(11)$
  is irrational, {\em Uspekhi Mat. Nauk}  (2001),  149--150.

\end{thebibliography}

\end{document}